\documentclass[10pt]{article}
\usepackage{amssymb}
\usepackage{amsmath}
\usepackage{amsthm}
\theoremstyle{plain}

\newtheorem{theorem}{Theorem}[section]
\newtheorem{proposition}[theorem]{Proposition}
\newtheorem{lemma}[theorem]{Lemma}
\newtheorem{corollary}[theorem]{Corollary}
\theoremstyle{definition}
\newtheorem{definition}[theorem]{Definition}
\newtheorem{example}[theorem]{Example}

\theoremstyle{remark}

\newcommand{\N}{\mathbb{N}}

\newcommand{\Z}{\mathbb{Z}}

\newcommand{\R}{\mathbb{R}}
\newcommand{\C}{\mathbb{C}}

\newcommand{\e}{\varepsilon}

\newcommand{\T}{\mathbb{T}}

\newcommand{\Zj}{{\cal Z}}
\newcommand{\Zjk}{{\cal Z}_{k}}
\newcommand{\Ot}{{\cal O}_2}

\newcommand{\im}{\sqrt{ -1}}
\newcommand{\tb}{\tilde{b}}

\newcommand{\te}{\tilde{e}}
\newcommand{\tf}{\tilde{f}}

\newcommand{\tc}{\tilde{c}}

\newcommand{\ts}{\tilde{s}}

\newcommand{\tT}{\widetilde{T}}
\newcommand{\tU}{\widetilde{U}}

\newcommand{\tsigma}{\widetilde{\sigma}}
\newcommand{\wtDelta}{\widetilde{\Delta}}

\newcommand{\wtu}{\widetilde{u}}

\newcommand{\wtv}{\widetilde{v}}

\newcommand{\wtz}{\widetilde{z}}
\newcommand{\wtU}{\widetilde{U}}
\newcommand{\tV}{\widetilde{V}}
\newcommand{\tW}{\widetilde{W}}

\newcommand{\cH}{{\cal H}}
\newcommand{\cR}{{\cal R}}

\newcommand{\od}{\overline{d}}
\newcommand{\ove}{\overline{e}}
\newcommand{\of}{\overline{f}}

\newcommand{\ov}{\overline{v}}

\newcommand{\ps}{{s'}}
\newcommand{\pv}{{v'}}

\newcommand{\bc}{\bar{c}}
\newcommand{\bs}{\bar{s}}
\newcommand{\bb}{\bar{b}}

\DeclareMathOperator{\id}{id}

\DeclareMathOperator{\Aut}{Aut}
\DeclareMathOperator{\Ad}{Ad}

\DeclareMathOperator{\spec}{sp}
\DeclareMathOperator{\Lip}{Lip}

\DeclareMathOperator{\Tr}{Tr}

\newcommand{\bprf}{\noindent{\it Proof.\ }}
\newcommand{\eprf}{\hspace*{\fill} \rule{1.6mm}{3.2mm} \vspace{1.6mm}}
\newcommand{\benu}{\begin{enumerate}\renewcommand{\labelenumi}{{\rm (\roman{enumi})}}\renewcommand{\itemsep}{0pt}}
\newcommand{\eenu}{\end{enumerate}}
\begin{document}
\title{The Rohlin property for automorphisms of the Jiang-Su algebra.}
\author{Yasuhiko Sato \\}
\date{\small Department of Mathematics, Hokkaido University,\\ Sapporo 060-0810, Japan\\e-mail : s053018@math.sci.hokudai.ac.jp}

\maketitle

\begin{abstract}   
For projectionless $C^*$-algebras absorbing the Jiang-Su algebra tensorially, we study a kind of the Rohlin property for autmorphisms. We show that the crossed products obtained by automorphisms with this Rohlin property also absorb the Jiang-Su algebra tensorially under a mild technical condition on the $C^*$-algebras. 

In particular, for the Jiang-Su algebra we show the uniqueness up to outer conjugacy of the automorphism with this Rohlin property. 


\end{abstract}

\section{Introduction}\label{INTRO}
In the classification program established by Elliott, the Jiang-Su algebra $\Zj$ is one of the most important $C^*$-algebras,  see  \cite{ET}, \cite{JS}, and which have been investigated by many people, \cite{DPT}, \cite{DW}, \cite{GJS}, \cite{Ror}, \cite{TW}.
Toms and Winter proved that for all approximately divisible $C^*$-algebras absorb the Jiang-Su algebra tensorially, i.e., $A \cong A\otimes \Zj$, \cite{TW}. R$\o$rdam showed that the Cuntz semigroup of a $\Zj$-absorbing  $C^*$-algebra is almost unperforated  \cite{Ror}. Recently, Winter has shown some criteria for the absorption of the Jiang-Su algebra  \cite{W2}. For abstract characterization of the Jiang-Su algebra in a streamlined way,  we refer to the recent papers by Dadarlat, R$\o$rdam, Toms, and Winter, \cite{DW}, \cite{RW}, \cite{W2}. H. Lin has shown the classification theorem for a large class of $C^*$-algebras consisting of generalized dimension drop algebras when they absorb the Jiang-Su algebra tensorially. 

In the case of von Neumann algebras Connes defined the Rohlin property for automorphisms, using a partition of unities consisting of projections, and classified automorphisms of injective type $I\hspace{-.1em}I_1$ factor up to outer conjugacy \cite{Con}. Kishimoto 
gave a method to prove the Rohlin property for automorphisms of AF-algebras for classifying automorphisms up to outer conjugacy,  based on the Elliott's classification program, \cite{EK}, \cite{Kis1}, \cite{Kis2}.  For Kirchberg algebras, Nakamura completely classified autmorphisms with the Rohlin property by their KK-classes up to outer conjugacy \cite{Nak}. Recently, Matui has classified automorphisms of AH-algebras with real rank zero and slow dimension growth up to outer conjugacy \cite{Mat2}. For finite actions, Izumi defined the Rohlin property and has shown the classification theory \cite{Iz2}, \cite{Iz3}. Recently, Izumi, Katsura, and Matui showed classification results for $\Z^2$-actions with the Rohlin property \cite{KM}, \cite{Mat}, \cite{IM}.

The aim of this paper is to introduce a kind of the Rohlin property for automorphisms of projectionless $C^*$-algebras and to give the following two main theorems as follows.

\begin{definition}\label{Defr}
Let $A$ be a unital $C^*$-algebra which has a unique tracial state $\tau$ and absorbs the Jiang-Su algebra $\Zj$ tensorially, and $\alpha$ be an automorphism of $A$. 
We say that $\alpha$ has {\it the\ weak\ Rohlin\  property}, if for any $k \in \N$ there exist positive elements $f_n \in A_+^1$ such that $(f_n)_n \in A_{\infty}$, $n\in \N$ 
\begin{eqnarray}
& &(\alpha^j(f_n))\cdot (f_n) =0, \quad j=1,2,...,k-1, \nonumber \\
& & \tau (1- \sum_{j=0}^{k-1} \alpha^j (f_n)) \rightarrow 0. \nonumber
\end{eqnarray}
\end{definition}
Here, we denote by $A^{\infty}$ the quotient $\ell^{\infty}(\N,A)/c_0(A)$, and $A_{\infty}$ the central sequence algebra $A^{\infty} \cap A'$. 

We extend a technical condition called property (SI) to projectionless $C^*$-algebras in Definition \ref{SI}. Roughly speaking, property (SI) means that if two central sequence of positive elements are given such that one of them is infinitesimally small compared to the other in the sequence algebra, then in fact so in the central sequence algebra.

\begin{theorem}\label{Thm1}
Let $A$ be a unital separable projectionless $C^*$-algebra which has a unique tracial state and absorbs the Jiang-Su algebra tensorially. Suppose that $A$ has property (SI) and  $\alpha $ is an automorphism of $A$ with the weak Rohlin property.
Then $A \times_{\alpha} \Z$ also absorbs the Jiang-Su algebra tensorially. 
\end{theorem}

\begin{theorem}\label{Thm2}
Suppose that $\alpha$ and $\beta$ are automorphisms of the Jiang-Su algebra with the weak Rohlin property. Then $\alpha$ and $\beta$ are outer conjugate, i.e.,
 there exists an automorphism $\delta$ of $\Zj$ and a unitary $u$ in $\Zj$
such that 
\[\alpha = \Ad u \circ \delta \circ \beta \circ \delta^{-1} .\]
\end{theorem}
For a separable, nuclear $C^*$-algebra $A$ absorbing the Jiang-Su algebra, R$\o$rdam proved that $A$ is purely infinite if and only if $A$ is traceless in \cite{Ror}, and Nakamura proved that the aperiodicity for automorphisms of purely infinite $C^*$-algebras coincides with the Rohlin property in \cite{Nak}. 
If $A$ is a projectionless $C^*$-algebra with a unique tracial state constructed in \cite{Sat2}, the weak Rohlin property is equivalent to the aperiodicity of the automorphism in the GNS-representation associated with the tracial state.  Similar definitions for finite actions, which is called {\it projection free tracial Rohlin property}, has been defined in  \cite{Arc}, \cite{OP}. The first main theorem is an adaptation of the result showed by Hirshberg and Winter in \cite{HW} to projectionless $C^*$-algebras. The second main theorem is an adaptation of the result for UHF algebras showed by Kishimoto in \cite{Kis1} to the Jiang-Su algebra.

The paper is organized as follows: 
In Section 2, we recall the generators of the prime dimension drop algebras  discovered by R$\o$rdam and Winter in \cite{RW}.
In Section 3, for projectionless cases we extend the technical property, which was called property (SI) in \cite{Sat}, to projectionless $C^*$-algebras. By this property, we can obtain the generators defined in Section 2 .
In Section 4, we prove Theorem \ref{Thm1}.
In Section 5, using the weak Rohlin property we show the stability for the automorphisms of the Jiang-Su algebra, and show Theorem \ref{Thm2}.

Concluding this section, we prepare some notations. When $A$ is a $C^*$-algebra, we denote by $A_{\rm sa}$ the set of self-adjoint elements of $A$, $A^1$ the unit ball of $A$, $A_+$ the positive cone of $A$, $U(A)$ the unitary group of $A$, $P(A)$ the set of projections of $A$, $T(A)$ the  tracial state space of $A$. 
We define an inner automorphism of $A$ by $\Ad u (a)$ $=uau^*$ for $u\in U(A)$ and $a\in A$. We denote by $M_n$  the $C^*$-algebra of $n\times n$ matrices with complex entries  and $e_{i,j}^{(n)}$ the canonical matrix units of $M_n$, and we set $e_i^{(n)} = e_{i,i}^{(n)}$.
We denote by $(m,n)$ the greatest common divisor of $m$ and $n\in\N$.

\section{The generators of prime dimension drop algebras}\label{Sec2}
The following argument was given by R$\o$rdam and Winter in \cite{Ror} and \cite{RW}. We would like to begin with some definitions about the generators of prime dimension drop algebras and show Proposition \ref{Prop1}. 
We denote by $I(k,k+1)$, $k\in \N$ the prime dimension drop algebra 
\[ \{f\in C([0,1])\otimes M_k \otimes M_{k+1} ;\ f(0) \in M_{k}\otimes1_{k+1},\ f(1)\in 1_k\otimes M_{k+1} \}, \]
and set the self adjoint unitary 
\[ u_1 = \sum_{i,j} e_{i,j}^{(k)} \otimes e_{j,i}^{(k)} \in U (M_k \otimes M_k).\]
Define non-unital $*$-homomorphisms $\rho_0 :M_k\otimes M_k \hookrightarrow M_k\otimes M_{k+1}$ by $\rho_0(e_{i,j}^{(k)}\otimes e_{l,m}^{(k)})= e_{i,j}^{(k)}\otimes e_{l,m}^{(k+1)}$,  and $\rho: C([0,1])\otimes M_k \otimes M_k \hookrightarrow C([0,1])\otimes M_k \otimes M_{k+1}$ by 
$ \rho(f)(t)= \rho_0(f(t))$, $t\in[0,1]$. 
Let $u \in U(C([0,1])\otimes M_k \otimes M_{k}) $ be such that $u(0) =1$ and $u(1) = u_1$, and set
\begin{eqnarray} 
v &=& \sum_{j=1}^k e_{1,j}^{(k)} \otimes e_{j,k+1}^{(k+1)},\nonumber \\
w(t) &=& \rho(u)(t) \oplus \cos^{1/2} (\pi t/2) 1_k\otimes e_{k+1}^{(k+1)},  \nonumber \\
c_j (t) &=& w(t)(e_{1,j}^{(k)}\otimes 1_{k+1}) w^*(t), \quad j=1,2,...,k, \nonumber \\
s  (t)&=& \sin (\pi t/2) w(t) v, \quad t\in [0,1]. \nonumber
\end{eqnarray}
Since $c_j (0) = e_{1,j}^{(k)} \otimes 1_{k+1}$, $c_j(1) = 1_k \otimes e_{1,j}^{(k+1)}$, $s (0) = 0$, and $s(1)=$ $1_k$$ \otimes$$ e_{1,k+1}^{(k+1)}$, it follows that $c_j$, $s \in I(k, k+1)$. 
And we have that 
\begin{eqnarray}
w^*w (t) &=& w w^* (t) = 1_k \otimes (1_{k+1}- e_{k+1}^{(k+1)}) \oplus \cos(\pi t/2) 1_k \otimes  e_{k+1}^{(k+1)}, \nonumber \\
c_ic_j^* &=& ww^*w (e_{1,i}^{(k)} e_{j,1}^{(k)} \otimes 1_{k+1} ) w^* =\delta_{i,j}c_1^2 , \nonumber \\
\sum_{j=1}^k c_j^* c_j &=& (w^*w)^2, \nonumber \\
s^* s (t) &=&  \sin^2(\pi t/2) 1_k \otimes e_{k+1}^{(k+1)}, \nonumber \\
c_1 s (t) &=& \sin( \pi t/2) w(e_1^{(k)} \otimes 1_{k+1}) w^* w (t) (\sum_{i=1}^k e_{1,i}^{(k)} \otimes e_{i, k+1}^{(k+1)}) = s (t). \nonumber \end{eqnarray}
From these computations, it follows that $\{ c_j \}_{j=1}^k \cup \{ s\}$ satisfies
\[ c_1 \geqq 0,\quad c_i c_j^* =\delta_{i,j} c_1^2,\] 
\[\sum_{j=1}^k c_j^* c_j + s^*s = 1 ,\quad c_1 s =s. \]
To be convenient, we denote by $\cR_k$ the above relations on generators of a unital $C^*$-algebra.

Fix a separable infinite-dimensional Hilbert space $\cH$, and set 
\[\Lambda=\{ \{c_j'\}_{j=1}^k \cup\{ s'\} \subset B(\cH)^1; {\rm \ satisfies\ } \cR_k\} \subset 2^{B(\cH)^1}. \]
For $\lambda \in \Lambda$, let $c_{j,\lambda}\in \lambda$, $j=$ $1,2,$ $...,k$ and $s_{\lambda}\in \lambda$ be generators corresponding to $c_j$, $j=$ $1,2,$ $...,k$, and $s$ on the relations $\cR_k$, and define $\tc_j=\bigoplus_{\lambda\in \Lambda}$ $c_{j,\lambda}$, $\ts=\bigoplus_{\lambda \in \Lambda}$ $s_{\lambda}$ $\in B(\bigoplus_{\lambda \in \Lambda} \cH) $. The set  $\{\tc_j\}_{j=1}^k$ $\cup \{\ts\}$
 satisfies the relations $\cR_k$. Let $C^*(\{ \tc_j\}_{j=1}^k$ $\cup \{\ts\})$ be the 
$C^*$-subalgebra of $B(\bigoplus_{\lambda \in \Lambda} \cH) $ generated by $\{\tc_j\}$ $\cup \{\ts \}$. Then, we can identify $C^*(\{ \tc_j\}_{j=1}^k$ $\cup \{\ts\})$  with the universal $C^*$-algebra on a set of generators satisfying the relations $\cR_k$.

\begin{proposition}[Proposition 5.1 in \cite{RW}]\label{Prop1}
The universal $C^*$-algebra $C^*$ $(\{\tc_j\}_{j=1}^k$ $ \cup$ $ \{ \ts\} )$ is isomorphic to $I(k,k+1)$ with 
$\tc_j \mapsto c_j$ and $\ts \mapsto s$. 
\end{proposition}
\begin{proof}
First we show that $C^*(\{c_j\}_j \cup \{s\}) = I(k,k+1)$. 
Since 
\begin{eqnarray}
\sum_{j=1}^k c_j^* ss^* c_j + s^* s(t) &=& \sin^2(\pi t/2) (1_k \otimes 1_{k+1}), \nonumber 
\end{eqnarray}
and $1_k \otimes 1_{k+1}$ $\in C^* (\{ c_j\}_j \cup \{ s \})$, we have that $C([0,1])\otimes 1_k \otimes 1_{k+1} \subset C^* (\{ c_j\}_j \cup \{ s \} )$. 
By a partition of unity argument on $[0,1]$, it suffices to show that 
$C^* (\{c_j \}_j \cup \{s \} ) (i) \cong M_{k+i}$, $ i=0,1$, and $C^* (\{ c_j \}_j \cup \{ s \} ) (t) \cong M_k \otimes M_{k+1}$, $t \in (0,1)$. 
Since $c_j (0)=$ $ e_{1,j}^{(k)}\otimes 1_{k+1}$, $c_j (1) = $ $1_k \otimes e_{1,j}^{(k+1)}$, $j=$ $1$, $2$, $...$, $k$, and $s (1) =1_k \otimes e_{1,k+1}^{(k+1)}$, it follows that $C^* (\{c_j \}_j \cup \{s \})$ $(i)$ $\cong M_{k+i}$, $i=0,1$. Since
\begin{eqnarray}
s c_j s ^* (t) &=& \sin^2(\pi t/2) \cos (\pi t /2) \rho (u) (e_{1,1}^{(k)} \otimes e_{1,j}^{(k+1)}) \rho (u)^*(t),\quad j=1,2,...,k, \nonumber \\
s c_j s^* c_i (t) &=& \sin^2(\pi t/2) \cos (\pi t /2) \rho (u)(e_{1,i}^{(k)} \otimes e_{1,j}^{(k+1)} ) \rho (u)^*(t),\ i,j=1,2,...,k, \nonumber \\
s c_j (t) &=& \sin(\pi t/2) \cos (\pi t /2) \rho (u)(t) (e_{1,j}^{(k)} \otimes e_{1,k+1}^{(k+1)}), \ j=1,2,...,k, \nonumber 
\end{eqnarray}
for $t \in (0,1)$,
 we have that $C^*(\{ c_j \}_j \cup \{ s \} ) (t) =$ $M_k \otimes M_{k+1}$ for $t \in (0,1)$.

Set $A$ $=$ $C^*(\{\tc_j\}_{j=1}^k\cup \{\ts\})$. Let $\Phi: A \rightarrow I(k,k+1)$ be the $*$-homomorphism defined by $\Phi (\tc_j) = c_j$ and $\Phi (\ts) = s$. It remains to show that $\Phi$ is injective. Let $(\pi, \cH)$ be an irreducible 
representation of $A$. Because for any $a\in A$ there exists an irreducible representation of $A$ which preserves the norm of $a$ (see \cite[4.3.10]{Ped}), it suffices to show that there exists a representation $\varphi$ of $I(k,k+1)$ on $\cH$
such that $\varphi(c_j)= \pi (\tc_j)$ and $\varphi (s) = \pi (\ts)$.

Set \[ \tb= \sum_{j=1}^k \tc_j^* \ts \ts^* \tc_j + \ts^* \ts . \] 
By the following computations, we see that $\tb$ is in the center of $A$. Since $\{\tc_j \}_j \cup \{ \ts \}$ satisfies the relations $\cR_k$, in particular $\tc_1^2 = \tc_j \tc_j^* $, we have that 
\begin{eqnarray}
\tb \tc_j &=& \ts \ts^* \tc_j + \ts^* \ts \tc_j, \nonumber \\
\tc_j \tb &=& \ts \ts^* \tc_j + \tc_j \ts^* \ts, \nonumber \\
\ts^* \ts \tc_j &=& \tc_j - \tc_1^2 \tc_j = \tc_j \ts^* \ts, \nonumber \\
\tb\ts &=& \tb\tc_1 \ts = \ts \ts^* \ts + \ts^* \ts^2, \nonumber \\
\ts \tb &=& \sum  \ts  \tc_j^*  \ts  \ts^*  \tc_j +  \ts  \ts^*  \ts, \nonumber \\
 \ts^*  \ts^2 &=&  \tc_1  \ts -  \tc_1^3  \ts =  \ts -  \ts =0, \nonumber \\
 \ts^*  \ts  \tc_j^*  \ts  \ts^*  \tc_j &=&  \tc_j^*  \ts  \ts^*  \tc_j -  \tc_j^*  \tc_j \tc_j^*  \ts  \ts^*  \tc_j =0. \nonumber
\end{eqnarray}
Hence $[ \tb,  \tc_j]=0$ and $[\tb,  \ts]=0$. 

Set $\bc_j=\pi(\tc_j)$, $\bs = \pi (\ts)$, and $\bb= \pi (\tb)$. 
Since $0\leqq \tb \leqq 1$, we have that $\spec (\tb)=[0,1]$ and obtain $\beta \in [0,1]$ such that $\beta 1=\bb$. When $\beta =0$  we have $\bs=0$. Thus $\{ \bc_j \}_j$ satisfies the relations for matrix units $\{ e_{1,j}^{(k)} \}$ of $M_k$, and then $\cH = \C^k$. Set $V_0: I(k,k+1) \rightarrow M_k $ as the irreducible representation at $0$. Since $V_0 (c_j)= e_{1,j}^{(k)}$ we obtain a unitary $u_0 \in U (M_k)$ such that $\Ad u_0 \circ V_0 (c_j ) = \bc_j $
and define $\varphi = \Ad u_0 \circ V_0$.  

When $\beta =1$ , by the following computations, we see that $\bc_j^* \bc_j$, $j=1,2,$ $...,n$, and $\bs^* \bs$ are orthogonal projections. 
Since $\bb=1$, we have that $\sum \bc_j^* (1- \bs \bs^*) \bc_j =0$. Then it follows that $\bc_1^2 = $ $\bs \bs^*$, $\bc_1^4=$ $\bc_1^2$,  $(\bc_j^* \bc_j)^2 = $ $\bc_j^* \bc_1^2 \bc_j=$ $\bc_j^* \bc_j$, and $(\bs^*\bs)^2 =$ $\bs^* \bc_1^2 \bs$ $= \bs^*\bs $. From $\sum \bc_j^* \bc_j$ $+ \bs^* \bs$ $=1$ it follows that $\bc_j^* \bc_j$, $j=1,2,...,k$ and $\bs^*\bs$ are mutually orthogonal projections. Hence $\{\bc_j\}_j \cup \{\bs\}$ satisfies the relations for matrix units  $\{ e_{1,j}^{(k+1)} \}_{j=1}^{k+1}$ of $M_{k+1}$.   Then we see that $\cH = \C^{(k+1)}$ and can define $\varphi: I(k,k+1) \rightarrow B(\cH)$ as the irreducible representation of $I(k,k+1)$ at $t=1$ up to unitary equivalence. 

When $0 < \beta <1 $, by the following computations,  we see that 
\begin{eqnarray}
E_{i,j} &=& (\beta (1- \beta))^{-1} \bc_i^* \bs \bc_j^* \bc_j \bs^* \bc_i, \nonumber \\
E_{j,k+1} &=& (\beta(1-\beta))^{-1} \bc_j^* \bc_j \bs^* \bs,\nonumber 
\end{eqnarray}
 $i,j=1,2,...,k$, are mutually orthogonal projections. Since $\bb\bs = \bs \bs^* \bs$ and $(1-\bb)\bs^* \bs = (1-\bs^* \bs) \bs^* \bs$, we have that
\begin{eqnarray}
E_{i,j}^2 &=& (\beta(1-\beta))^{-2} \bc_i^* \bs \bc_j^* \bc_j \bs^* \bc_1^2 \bs \bc_j^* \bc_j \bs^* \bc_i \nonumber \\
 &=& (\beta(1-\beta))^{-2} \bc_i^* \bs \bs^* \bs \bc_j^* \bc_j \bc_j^* \bc_j \bs^* \bc_i \nonumber \\
&=& \beta^{-1} (1-\beta)^{-2} \bc_i^* \bs (1-\bs^* \bs) \bc_j^*\bc_j \bs^* \bc_i = E_{i,j}, \nonumber \\
E_{j,k+1}^2 &=& (\beta (1- \beta))^{-2} (\bc_j^* \bc_j)^2 (\bs^* \bs)^2 \nonumber \\
&=& \beta^{-1}(1-\beta)^{-2} \bc_j^* \bc_j (1- \bs^* \bs) \bs^* \bs =E_{j,k+1}, \nonumber \\
\sum_{i,j} E_{i,j}+ \sum_j E_{j,k+1} &=& (\beta(1-\beta))^{-1} (\sum_{i=1}^k \bc_i^* \bs (1-\bs^* \bs) \bs^* \bc_i +(1- \bs^* \bs) \bs^* \bs) \nonumber \\
&=& \beta^{-1} (\sum_{i=1}^k \bc_i^* \bs \bs^* \bc_i + \bs^* \bs) =1. \nonumber
\end{eqnarray}

Set 
\begin{eqnarray}
& &F_{i,j}= \beta^{-1} (1- \beta)^{-1/2} \bs \bc_j \bs^* \bc_i, \nonumber \\
& &F_{j, k+1} = (\beta (1-\beta))^{-1/2} \bs \bc_j,\quad i,j=1,2,...,k. \nonumber 
\end{eqnarray}
Then it follows that $F_{i,j}^* F_{i,j} $ $= E_{i,j}$, $F_{j,k+1}^* F_{j,k+1} $ $=E_{j,k+1}$, $F_{i,j}F_{i,j}^*=E_{1,1}$, and $F_{j,k+1} F_{j,k+1}^*$ $= E_{1,1}$.
 Thus $\{ F_{i,j} \}_{i,j} \cup \{ F_{j,k+1} \}_j$ satisfies the same relations as matrix units $\{ e_{1,i}^{(k)} \otimes e_{1,j}^{(k+1)} \}_{i,j} $ $\cup$ $\{e_{1,j}^{(k)} \otimes e_{1,k+1}^{(k+1)} \}_j$ of $M_k \otimes M_{k+1}$. It is not so  hard to see that 
$\beta^{1/2}\sum_{j=1}^k F_{1,j}^*F_{j,k+1}$ $=\bs$, $\beta^{1/2} \sum_{j=1}^k F_{1,j}^* F_{i,j} =$ $\bs \bs^* \bc_i$, and $\beta \bc_i= $ $\bs \bs^* \bc_i+$ $  \bs^* \bs \bc_i $. Then we have that $C^* (\{F_{i,j}\} \cup \{F_{j,k+1} \} ) =$ $ C^* ( \{ \bc_j\}_j \cup \{ \bs \})$ and  $\cH = \C^{k(k+1)}$. Let $V_{\beta}$ be the irreducible representation of $I(k,k+1)$ at $t$ $\in(0,1)$ with $\sin^{2}(\pi t/2)=\beta$. Then $V_{\beta}\circ\Phi (\tb)=\beta$ and there exists a unitary $u_{\beta}$ such that 
\begin{eqnarray}
F_{i,j} = \beta^{-1} (1-\beta)^{-1/2} \Ad u_{\beta} \circ V_{\beta} (s c_j s^* c_i), \nonumber \\
F_{j,k+1}= (\beta (1-\beta))^{-1/2}\Ad u_{\beta} \circ V_{\beta} (s c_j). \nonumber 
\end{eqnarray}
Hence, we have that $\Ad u_{\beta} \circ V_{\beta} (c_j) = \bc_j$ and $\Ad u_{\beta} \circ V_{\beta} (s) = \bs $ and obtain $\varphi = \Ad u_{\beta} \circ V_{\beta}$. This completes the proof.    
\end{proof}

\begin{corollary}\label{Cor2.1}
Let $A$ be a unital $C^*$-algebra. Suppose that $A^1$ contains $\bc_i$, $i=$ $1,2,...,k$, and $\bs$ satisfying the relations $\cR_k$ and $\spec (|\bs|) = [0,1] $. 
Then there exists a unital embedding $\varphi : I(k,k+1) \hookrightarrow A$ such that $\varphi (c_i) = \bc_i$ and $\varphi (s) = \bs$.  
\end{corollary}
\begin{proof}
By the universal property of $I(k,k+1)$,
 there exists a unital $*$ homomorphism $\varphi$ $:I(k,k+1)$ $\rightarrow A$ such that $\varphi (c_j) = \bc_j$ and $\varphi (s) =\bs $.  We denote by $X_{\varphi}$ the compact subset of $[0,1]$ determined by $ \ker (\varphi) =$ $\{ f \in I(k,k+1) ; f|_{X_{\varphi}} =0 \}$. For any $f\in C([0,1])$ it follows that $f(|s|) (t)=$ $f\circ \sin(\pi t/2 )$ $ 1_{k}\otimes e_{k+1}^{(k+1)}$. Then we have that $f|_{\sin(\frac{\pi}{2} X_{\varphi})} =0 $ $\Leftrightarrow $ $f(|s|) |_{X_{\varphi}}=0$ $\Leftrightarrow$ $f(|\bs|)=0$ $\Leftrightarrow$ $f|_{\spec (|\bs|)}=0$. Thus $\spec (|\bs|) = \sin (\frac{\pi}{2} X_{\varphi})$, which concludes that $\ker (\varphi) =  0 $. \end{proof}

\section{Property (SI) for projectionless $C^*$-algebras}
\begin{definition}
Let A be a unital $C^*$-algebra and  $\tau\in T(A)$. 
We recall the dimension function $d_{\tau}$ and define $\od_{\tau} : A_+^1 \rightarrow \R_+$ by  
\begin{eqnarray}
 d_{\tau} (f) &=& \lim_{n\rightarrow \infty}  \tau( (1/n + f)^{-1} f),  \nonumber \\
 \od_{\tau} (f) &=& \lim_{n \rightarrow \infty} \tau (f^n),\quad f \in A_+^1. \nonumber \end{eqnarray}
\end{definition}

\begin{lemma}\label{Lem3.2}
For $f_n \in A_+^1$, $n\in \N$ with $ (f_n)_n \in A_{\infty}$ and an increasing sequence $m_n$, $n\in \N$ with $m_n\nearrow \infty$, it follows that 
\begin{description}
\item[(1)]\label{(1)}
If $\displaystyle \lim_{n\rightarrow \infty} \max_{\tau \in T(A)} \tau (f_n) =0 $ then there exists $\tf_n \in A_{+}^1$, $n\in \N$ such that $(\tf_n)_n = (f_n)_n$ and $\displaystyle \lim_{n\rightarrow \infty } \max_{\tau \in T(A)}$ $d_{\tau} (\tf_n) = 0 $.
\item[(2)]\label{(2)}
There exists $\tf_n\in $ $A_+^1$, $n\in \N$ such that $(\tf_n )_n= (f_n)_n $ and $\displaystyle \liminf_{n\rightarrow \infty} \min_{\tau \in T(A)}$ $\od_{\tau} (\tf_n)$ $\geq$ $\liminf_n \min_{\tau} \tau (f_n^{m_n})$.
\item[(3)]\label{(3)}
If $A$ absorbs $\Zj$ tensorially, then there exists $f_n^{(i)}\in A_+^1$, $i=0,1$, $n\in \N$ such that $(f_n^{(i)})_n \in A_{\infty}$, $f_n^{(0)} f_n^{(1)}=0$, $(f_n^{(i)})_n \leq (f_n)_n$, $i=0,1$, and $\displaystyle \liminf_{n\rightarrow \infty} \min_{\tau \in T(A)}$ $\od_{\tau} (f_n^{(i)})$ $\geq$ $\liminf_n \min_{\tau} \tau (f_n^{m_n})/2$.
\end{description}
\end{lemma}
\bprf
(1) Let $\varepsilon_n >0$ be such that $\varepsilon_n \searrow 0$ and $\max_{\tau \in T(A)} \tau (f_n) \leq \varepsilon_n^2$. Set 
\begin{eqnarray}
g_{\varepsilon}(t)=\left\{ \begin{array}{ll}
(1 -\varepsilon)^{-1} (t-\varepsilon) ,\quad & \varepsilon \leq t \leq 1, \\
0, \quad  & 0\leq t \leq \varepsilon, \\
\end{array} \right.\nonumber
\end{eqnarray} 
and $\tf_n= g_{\varepsilon_n} (f_n)$. Then we have that 
$\| \tf_n - f_n \| \leq \varepsilon _n $ and $\varepsilon_n \lim_{m\rightarrow \infty } (1/m + \tf_n)^{-1} \tf_n \leq f_n$, which implies that $d_{\tau} (\tf_n) \leq \varepsilon_n$, for any $\tau \in T(A)$.

(2) Let $\varepsilon_n >0$ be such that $\varepsilon_n \searrow 0$, and $(1-\varepsilon_n)^{m_n} \rightarrow 0 $. Set
\begin{eqnarray}
g_{\varepsilon}(t)=\left\{ \begin{array}{ll}
(1 -\varepsilon)^{-1} t,\quad & 0 \leq t \leq 1-\varepsilon, \\
1, \quad  & 1-\varepsilon \leq t \leq 1, \\
\end{array} \right.\nonumber
\end{eqnarray} 
and $\tf_n = g_{\varepsilon_n} (f_n)$. Then we have that $\| \tf_n - f_n \| \leq \varepsilon_n$ and $f_n^{m_n} =$ $f_n^{m_n} (\lim_{l \rightarrow \infty}$ $\tf_n^l + \chi ([0,1-\varepsilon_n))(f_n))$ $\leq \lim_{l \rightarrow \infty} \tf_n^l$ $+ (1- \varepsilon_n )^{m_n}$, (where $\chi(S)$ means the characteristic function of $S$,) which implies that  $\tau (f_n^{m_n})\leq $ $\od_{\tau} (\tf_n) + (1-\varepsilon_n)^{m_n}$, for any $\tau \in T(A)$.
 
(3) Set $c=$ $\displaystyle \liminf_{n\rightarrow \infty } \min_{\tau\in T(A)} \tau(f_n^{m_n})$.  Since $A\cong$ $\displaystyle A\otimes \bigotimes_{n\in \N} \Zj$, we obtain $l_n \in \N$ and $\of_n\in $ $\displaystyle (A\otimes \bigotimes_{j=1}^{l_n} \Zj)_+^1$ such that $l_n \nearrow \infty $ and $m_n\|\of_n - f_n\|$ $\rightarrow 0$, which implies that $(\of_n)_n \in A_{\infty}$ and $\displaystyle \liminf_{n\rightarrow \infty} \min_{\tau \in T(A)} \tau (\of_n^{m_n}) = c$. 
By an argument as in the proof of (2), we obtain $\tf_n \in$ $\displaystyle (A\otimes \bigotimes_{l=1}^{l_n} \Zj)_+^1$, $n\in \N$ such that $(\tf_n)_n = (\of_n)= (f_n)$ and $\displaystyle \liminf_{n\rightarrow \infty} \min_{\tau\in T(A)} \od_{\tau} (\tf_n) \geq c$.
Let $g_n^{(i)} \in \Zj_+^1$, $i=0,1$, $n\in \N$ be such that $g_n^{(0)}g_n^{(1)}=0$, $\liminf_n \od_{\tau_{\Zj}} (g_n^{(i)}) = 1/2$, $i=0,1$, where $\tau_{\Zj}$ means the unique tracial state of $\Zj$. Set $f_n^{(i)}=\tf_n \otimes g_n^{(i)} \in A\otimes \bigotimes_{l=1}^{l_n+1} \Zj $. 
Since $l_n \nearrow \infty$, it follows that $(f_n^{(i)})_n\in A_{\infty}$, and since $\tau ((f_n^{(i)})^p) = \tau (\tf_n^p \otimes 1)\tau_{\Zj}((g_n^{(i)})^p)$, $p\in \N$, $\tau \in T(A)$, it follows that $\liminf_n \min_{\tau} \od_{\tau}(f_n^{(i)}) $ $=\liminf_n \min_{\tau} \od_{\tau} (\tf_n) \od_{\tau_{\Zj}} (g_n^{(i)}) \geq c/2$, $i=0,1$. 
\eprf

In \cite{Sat}, we have defined the technical condition, called property (SI), for $C^*$-algebra with non trivial projections. For projectionless $C^*$-algebras, we generalize this technical condition in the following.
\begin{definition}\label{SI}
We say that $A$ has the {\it property (SI) },
 when for any $e_n$ and $f_n \in A_+^1$, $n\in\N$ satisfying the following conditions:
$(e_n)_n, (f_n)_n \in A_{\infty}$, 
\begin{description}\label{prop:embedding2}
\item \[ \lim_{n\rightarrow \infty} \max_{\tau\in T(A)} \tau (e_n) =0, \] 
\item
\[ \liminf_{n\rightarrow \infty} \min_{\tau \in T(A)}  \tau (f_n^n) > 0,\] \end{description} 
 there exists $s_n\in A^1$, $n\in \N$, such that  $(s_n)_n \in A_{\infty}$ and 
\[(s_n^* s_n) = (e_n),\quad (f_n s_n) = (s_n).\]
\end{definition}

\begin{example}\label{UHFSI}
Any UHF algebra has the property (SI).
\end{example}
\bprf
Let $B$ be a UHF algebra, and let $e_n$ and $f_n\in B_+^1$, $n\in \N$ satisfy the conditions in the property (SI). Let $B_n$, $n\in \N$ be an increasing sequence of matrix subalgebras of $B$ such that $\overline{(\bigcup_{n\in\N} B_n)} =B$ and $1_{B_n}=1_B$. For any $B_{n}$, we denote by $\Phi_n : B \longrightarrow B_n'\cap B$ the conditional expectations (\cite[Section 1]{Sat}).
By $(e_n)_n$, $(f_n)_n$ $\in B_{\infty}$, we obtain a slow increasing sequence $m_n$ $\in\N$, $n$ $\in\N$ such that $m_n \nearrow \infty $, $m_n\leq n$, $(\Phi_{m_n} (e_n))_n =(e_n)_n$, and 
\[ \lim_{n\rightarrow \infty} m_n\| \Phi_{m_n} (f_n) -f_n \| =0,\]
and we obtain 
a fast increasing sequence $l_n$, $n\in\N$, and $\ove_n$, $\of_n$ $\in (B_{m_n}'\cap B_{l_n})_+^1$ such that  $m_n<l_n$, $(\ove_n)_n=(\Phi_{m_n}(e_n))_n$, and  $\lim_{n\rightarrow \infty}$ $m_n \| \of_n$ $- \Phi_{m_n} (f_n) \| =0$.
Then we have that $\lim \tau (\ove_n)=0$ and  $\lim \| \of_n^{m_n} -f_n^{m_n} \| =0$, which implies that 
$\liminf \tau(\of_n^{m_n})$ $= \liminf \tau(f_n^{m_n})$ $>0$.
By Lemma 3.2 (1), we obtain $\te_n \in (B_{m_n}'\cap B_{l_n})_+^1$ such that $(\te_n)_n=(e_n)_n$ and $\lim d_{\tau} (\te_n)=0$. By Lemma 3.2 (2), we obtain $\tf_n \in (B_{m_n}'\cap B_{l_n})_+^1$ such that $(\tf_n)_n= (f_n)_n$ and 
\[\liminf_{n\rightarrow\infty} \od_{\tau} (\tf_n) \geq \liminf \tau (\of_n^{m_n})>0.\]

Taking a large $N \in \N$, we have that 
\[ d_{\Tr_n} (\te_n) = d_{\tau}(\te_n) < \od_{\tau}(\tf_n) = \od_{\Tr_n} (\tf_n),\ n\geq N,\]
where $\Tr_n$ is the normalized trace of $B_{m_n}' \cap B_{l_n}$.
Then, we obtain $s_n \in (B_{m_n}' \cap B_{l_n})^1$ such that 
$s_n^* s_n$ $= \te_n$, $\tf_n s_n= s_n$,
hence we have that  $(s_n)_n\in B_{\infty}$, $(s_n^* s_n)_n =(e_n)_n$, and $(f_n s_n)_n =(s_n)_n$. 
\eprf
 
By the above example and the following proposition, we see that the Jiang-Su algebra has the property (SI). This proposition is motivated by Lemma 3.3 in \cite{Mat2}.  

\begin{proposition}\label{Prop3.4}
Let $A$ be a unital $C^*$-algebra absorbing the Jiang-Su algebra $\Zj$ tensorially. If $A\otimes B$ has the property (SI) for any UHF algebra $B$, then $A$ also has the property (SI).
\end{proposition}

In order to prove the above proposition, we define the projectionless $C^*$-algebra $\Zjk$ for $k\in \N \setminus \{ 1\}$ by 
\begin{eqnarray}
\Zjk=& &\{ f\in C([0,1])\otimes M_{k^\infty}\otimes M_{(k+1)^{\infty}};\nonumber \\
 & &f(0)\in M_{k^{\infty}}\otimes 1_{{k+1}^{\infty}},\ f(1)\in 1_{k^{\infty}}\otimes M_{{k+1}^{\infty}} \}.\nonumber 
\end{eqnarray}
This projectionless $C^*$-algebra $\Zjk$ was introduced by R$\o$rdam and Winter in \cite{RW}, \cite{W1} as a mediator between $C^*$-algebras absorbing UHF algebras and $C^*$-algebras absorbing the Jiang-Su algebra. 

\bprf
 Suppose that $e_n$ and $f_n \in A_+^1$, $n\in \N$ satisfy the conditions in the property (SI).
Let $k$ be a natural number with $k\geq 2$, $B^{(i)}$ the UHF algebra of rank $(k+i)^{\infty}$, and $\Phi^{(i)}$  the canonical unital embeddings of $A \otimes B^{(i)}$ into $A\otimes B^{(0)} \otimes B^{(1)}$,  $i=0,1$. 

By Lemma \ref{Lem3.2} there exist $f_n^{(i)} \in A_+^1$, $i=0,1$, $n\in \N$, such that $(f_n^{(i)})_n \in A_{\infty}$, $(f_n^{(0)})_n (f_n^{(1)})_n =0 $, $(f_n^{(i)})_n  \leq   (f_n)_n$, and $\liminf_{n\rightarrow \infty}\min_{\tau\in T(A)} \tau( {f_n^{(i)}}^n)$ $\geq $ $\liminf_{n} \min_{\tau } \od _{\tau} (f_n^{(i)})$ $\geq$ $ \liminf_n\min_{\tau} \tau (f_n^n) /2$ $>0 $, $i=0,1$. 

Applying the property (SI) of $A\otimes B^{(i)} $ to $e_n \otimes 1_{(k+i)^{\infty}} $ and $ f_n^{(i)}\otimes 1_{(k+i)^{\infty}}$ $\in {A\otimes B^{(i)}}_+ ^1$ we obtain $s_n^{(i)} \in$ ${A\otimes B^{(i)} }^1$, $i=0,1$, $n\in \N$ such that $(s_n^{(i)})_n \in (A\otimes B^{(i)} )_{\infty}$,
\[ ({s_n^{(i)}}^*s_n^{(i)})_n = (e_n \otimes 1_{(k+i)^{\infty}}),\quad (f_n^{(i)} \otimes 1_{(k+i)^{\infty}} \cdot s_n^{(i)})_n = (s_n^{(i)}). \]
Note that $(\Phi^{(i)} ({s_n^{(i)}}^* s_n^{(i)}))_n$ $= (e_n \otimes 1_{k^{\infty}} \otimes 1_{(k+1)^{\infty}})_n$, $(f_n \otimes 1_{k^{\infty}} \otimes 1_{(k+1)^{\infty}})_n $ $\cdot (\Phi^{(i)} (s_n^{(i)}))_n$ $= ( \Phi^{(i)} (s_n^{(i)}))_n$
, and $(\Phi^{(0)} (s_n^{(0)}))_n ^* $ $ (\Phi^{(1)} (s_n^{(1)}))_n = 0 $. 

Define $s_n \in {A\otimes \Zjk }^1$, $n\in \N$ by
\[ s_n (t) = \cos (\pi t/2) \Phi^{(0)} (s_n^{(0)}) + \sin (\pi t /2) \Phi^{(1)} (s_n^{(1)}),\quad t\in [0,1]. \]
Since  
\begin{eqnarray}
(s_n^* s_n (t))_n &=&(\cos^2 (\pi t /2 ) \Phi^{(0)} ({s_n^{(0)}}^* s_n^{(0)}) + \sin^2 (\pi t/ 2) \Phi^{(1)}({s_n^{(1)}}^*s_n^{(1)}) \nonumber \\
&+& \cos \cdot \sin (\pi t /2) (\Phi^{(0)} (s_n^{(0)})^* \Phi^{(1)}(s_n^{(1)}) + \Phi^{(1)} (s_n^{(1)})^* \Phi^{(0)} (s_n^{(0)})))_n \nonumber \\
&=& (e_n \otimes 1_{k^{\infty}} \otimes 1_{(k+1)^{\infty}})_n, \quad t\in [0,1], \nonumber 
\end{eqnarray}
$(f_n\otimes 1_{k^{\infty}}\otimes 1_{(k+1)^{\infty}})_n$ $(s_n (t)) =$ $(s_n (t))$, $t\in [0,1]$, and $\Lip (s_n) = \pi$, $n\in \N$, it follows that 
\[ (s_n^* s_n)_n = (e_n \otimes 1_{\Zjk})_n, \quad (f_n\otimes 1_{\Zjk} )_n (s_n)_n = (s_n)_n. \]

Set $\iota: A_{\infty} \hookrightarrow (A\otimes \Zjk)_{\infty}$ by $\iota ((a_n)_n)= (a_n\otimes 1_{\Zjk})_n$. Since $\displaystyle A\cong$ $A\otimes \bigotimes_{n\in \N}$ $\Zj$ and $\Zjk \subset_{\rm unital} \Zj $, for any finite subset $F\subset A_{\infty}$, we obtain a unital embedding $\Phi_{F}:(A\otimes \Zjk)_{\infty} \hookrightarrow A_{\infty}$ such that $\Phi_{F} \circ \iota (x) = x$, $x\in F$.

Define $s= \Phi_{\{(e_n), (f_n)\}} ( (s_n)_n) \in A_{\infty}$, then we conclude that $s^*s =(e_n)$ and $(f_n)s=s$.
\eprf

\section{$\Zj$ absorption of Crossed products.}
In this section we  prove Theorem \ref{Thm1}.
We denote by $A_{\alpha}$ the fixed point algebra of $\alpha \in \Aut (A)$ and by $\alpha_{\infty}$ the automorphism of $A_{\infty}$ induced by $\alpha$. In the following Lemma \ref{Lem4.1}, mimicking Theorem 4.4 in \cite{HW}, we use the weak Rohlin property to obtain a set of elements in $(A_{\infty})_{\alpha_{\infty}}$ which satisfies the same relations as $\{ c_j\}_{j=1}^k$ in $\cR_k$. After that, applying the property (SI) and the weak Rohlin property, we obtain the generators of prime dimension drop algebras in $({A_{\infty}})_{\alpha_{\infty}}$.
\begin{lemma} \label{Lem4.1}
Let $A$ be a unital separable simple $C^*$-algebra which has a unique tracial state $\tau$ and absorbs the Jiang-Su algebra tensorially. Suppose that $\alpha\in \Aut (A)$ has the weak Rohlin property. 
Then for any $k \in \N$ there exist $c_{j,n} \in A$, $j=$ $1,2,...,k$, $n\in \N$ such that $(c_{j,n})_n\in (A_{\infty})_{\alpha_{\infty}}^1$, 
\[ (c_{1,n})_n \geqq 0,\quad (c_{i,n})_n(c_{j,n})^* =\delta_{i,j} (c_{1,n})^2, \]
\[\| (c_{1,n})_n \| =1 , \quad \|1-\sum_{j=1}^k (c_{j,n})_n^* (c_{j,n}) \| =1 ,\quad \lim_{n \rightarrow \infty} \tau (c_{1,n}^n) = 1/k .\]
(which implies $\lim_{n \rightarrow \infty} \tau (1_A-\sum_{j=1}^k c_{j,n}^*c_{j,n}) =0$.)
\end{lemma}

\bprf
Let $\Phi_m$, $m\in \N$ be the unital embeddings of $\Zj$ into $A\otimes \bigotimes_{m\in \N} \Zj \cong A$ defined by 
\[ \Phi_m (x) = 1_{A} \otimes 1_{\bigotimes_{i=1}^m \Zj} \otimes x \otimes 1_{\bigotimes_{i=m+2}^{\infty} \Zj},\quad x\in \Zj,\]
and $\Phi$ be the unital embedding of $\Zj$ into $A_{\infty}$ defined by $\Phi (x) = (\Phi_m (x))_m$, $x\in \Zj$. Remark that $\tau (\Phi_m (x)) = \tau_{\Zj} (x)$, $m\in \N$, $x\in \Zj$. 

In the definition of $c_j\in I(k,k+1)$ in Section 2, replacing $\cos$ with $\cos^{1/n^2}$ we obtain $c_n^{(j)}\in I(k,k+1) \subset \Zj$, $j=1,2,...,k$ such that
\[ c_n^{(1)} \geqq 0, \quad c_n^{(i)}{c_n^{(j)}}^* =\delta_{i,j} {c_n^{(1)}}^2, \]
\[\|(c_n^{(1)})_n \| =1,\quad \| 1- \sum_{j=1}^k {c_n^{(j)}}^* c_n^{(j)} \| =1, \quad  \tau_{\Zj} ({c_n^{(1)}}^n) \nearrow 1/k .\]
Let $\varepsilon_n > 0 $, $n\in \N$ be such that $\varepsilon_n \searrow 0$ and  $\tau_{\Zj}({c_n^{(1)}}^n)$ $>$ $1/k -\varepsilon_n$, $n\in \N$. 

Let $k_n$ and $l_n\in \N$, $n\in \N$ be such that $l_n \nearrow \infty$ and $l_n^2 < k_n$. 
Since $\alpha\in \Aut (A)$ has the weak Rohlin property and $A$ and $\Zj$ are separable $C^*$-algebras, we obtain $f^{(n)} =$ $ (f_m^{(n)})_m$ $\in$ $(A\cup $ $\bigcup_{j\in \Z} \alpha_{\infty}^j (\Phi (\Zj)))'$ $\cap A^{\infty}$  such that $\| f^{(n)} \|$ $=$ $1$, and 
\begin{eqnarray}
& & \alpha_{\infty}^p (f^{(n)}) f^{(n)} =0 ,\quad  0< |p| \leq 2(k_n+l_n),  \nonumber\\
& & \tau ({f_r^{(n)}}^n) > 1/(2(k_n+l_n) +1)- \varepsilon_m,\quad r \geq m. \nonumber
\end{eqnarray}
Note that any subsequence of $(f_m^{(n)})_m$ satisfies the above conditions. Then, taking a subsequence of $(f_m^{(n)})_m$, we may suppose that 
\[ |\tau (\Phi_m({c_n^{(1)}}^n)\cdot {f_m^{(n)}}^n) - \tau_{\Zj} ({c_n^{(1)}}^n)\tau({f_m^{(n)}}^n)| < \varepsilon_m .\]

For $p\in \Z$, define $a_{p,n}\geq 0$ by 
\begin{eqnarray}
a_{p,n}=\left\{ \begin{array}{lll}
1 -(|p|-k_n)/l_n,\quad & k_n < |p| \leq k_n+l_n, \\
1, \quad  &   |p| \leq k_n , \\
0, \quad & k_n +l_n < |p|, \\
\end{array} \right.\nonumber
\end{eqnarray} 
and define completely  positive maps $\varphi_n :$ $\Zj \longrightarrow A_{\infty}$, by 
\[\varphi_n (x) = \sum_{|p|\leq k_n+l_n} a_{p,n} \alpha_{\infty}^p (\Phi (x)) \alpha_{\infty}^p (f^{(n)}) .\]
Then we have that $\| \alpha_{\infty} (\varphi_n (x)) - \varphi_n (x)  \|$ 
\begin{eqnarray}
&=&\| \sum_{|p| \leq k_n +l_n} (a_{p,n} - a_{p-1,n})\cdot \alpha_{\infty}^p (\Phi (x)) \alpha_{\infty}^p (f^{(n)}) \| \nonumber \\
&=& \| x \| /l_n, \quad x\in \Zj, \nonumber
\end{eqnarray}
\begin{eqnarray}
\varphi_n(c_n^{(i)})\varphi_n(c_n^{(j)})^* &=& \sum_{|p|\leq k_n + l_n} a_{p,n} ^2 \alpha_{\infty}^p (\Phi (c_n^{(i)}{c_n^{(j)}}^*)) \alpha_{\infty}^p (f^{(n)})^2 \nonumber \\
&=& \delta_{i,j} \varphi_n(c_n^{(1)})^2,\quad n\in \N, \nonumber
\end{eqnarray} 
$\|\varphi_n(c_n^{(1)})\| =1$, and $ \|1$ $-$ $\sum_{j=1}^k \varphi_n (c_n^{(j)})^*\varphi_n(c_n^{(j)})$ $\|=1$.

Let $c_{n,m}^{(j)}\in A^1$, $j=1$,$2$,...,$k$, $m\in\N$ be components of $\varphi_n(c_n^{(j)})$ (i.e., $(c_{n,m}^{(j)})_m$ $=$ $\varphi_n (c_n^{(j)}) \in A_{\infty}$) with $c_{n,m}^{(1)}$ $\geqq 0$, then we have that 
\begin{eqnarray}
\liminf_{m\rightarrow \infty} \tau ({c_{n,m}^{(1)}}^n)  &=& \liminf_m \sum_{|p|\leq k_n+l_n} a_{p,n}^n \tau(\alpha ^p(\Phi_m (c_n^{(1)})^n) \alpha^p (f_m^{(n)})^n) \nonumber \\
 &>& (2k_n+1) \liminf_m \tau(\Phi_m({c_n^{(1)}}^n) \cdot {f_m^{(n)}}^n) \nonumber \\
&=& (2k_n +1) \liminf_m \tau_{\Zj} ({c_n^{(1)}}^n) \tau ({f_m^{(n)}}^n) \nonumber \\ 
&>& \frac{(2k_n+1)}{(2(k_n+l_n) +1)} (1/k -\varepsilon_n) \rightarrow_{n\rightarrow \infty} 1/k, \nonumber 
\end{eqnarray}

Let $F_n$ be an increasing sequence of finite subsets of $A^1$ with $\overline{\bigcup_n F_n}= A^1 $. By the above conditions, we obtain an increasing sequence $m_n\in \N$, $n\in \N$ such that 
\begin{eqnarray}
& &\|[c_{n,m_n}^{(j)}, x ] \| < \varepsilon ,\quad x\in F_n, \nonumber\\
& &\| \alpha (c_{n,m_n}^{(j)}) - c_{n,m_n}^{(j)} \| < 1/l_n +\varepsilon_n, \nonumber \\
& &\| c_{n,m_n}^{(i)}{c_{n,m_n}^{(j)}}^* -\delta_{i,j} {c_{n,m_n}^{(1)}}^2 \| <\varepsilon_n,\ i,j=1,2,...,k,  \nonumber \\
& &\|c_{n,m_n}^{(1)} \| > 1 - \varepsilon_n,\quad \| 1- \sum_{j=1}^k {c_{n,m_n}^{(j)}}^* c_{n,m_n}^{(j)} \| > 1 -\varepsilon_n, \nonumber \\
& &\tau({c_{n,m_n}^{(1)}}^n) > \frac{2k_n +1}{2(k_n+l_n) +1} (1/k -\varepsilon_n). \nonumber 
\end{eqnarray}
Define $c_{j,n} = c_{n,m_n}^{(j)}$, $j=$ $1,2$ $,...,k$, then we have that 
$(c_{j,n})_n \in (A_{\infty})_{\alpha_{\infty}}$, 
\[(c_{1,n})_n \geqq 0,\quad (c_{i,n})_n(c_{j,n})^* = \delta_{i,j}(c_{1,n})^2,\]
\[\|(c_{1,n})_n \| =1, \quad \| 1- \sum_{j=1}^k (c_{j,n})_n^* (c_{j,n})\| = 1 , \quad \lim_{n\rightarrow \infty} \tau(c_{1,n}^n) = 1/k .\]\eprf

By the technique in the proof of Lemma 4.6 \cite{Kis2} we obtain a generator  $s$ in $(A_{\infty})_{\alpha_{\infty}}$ satisfying the relations in $\cR_k$ together with $\{(c_{j,n})_n\}$ above. In the proof of the following proposition, $x\approx_{\varepsilon} y$ means $\| x-y\| < \varepsilon $. 

\begin{proposition}\label{Prop4.2}
Let $A$ be a unital $C^*$-algebra which has a unique tracial state $\tau$, absorbs the Jiang-Su algebra $\Zj$ tensorially, and has the property (SI). 
Suppose that $\alpha \in \Aut (A)$ has the weak Rohlin property. 
Then for any $k\in \N$ there exists a set of norm-one elements $\{c_j\}_{j=1}^k \cup \{s\}$ in $(A_{\infty})_{\alpha_{\infty}} $ satisfying $\cR_k$.
\end{proposition}
\bprf
By Lemma \ref{Lem4.1} we obtain $c_m^{(j)} \in A ^1$, $j=$ $1,2,...,k$, $m\in\N$ such that $(c_m^{(j)})_m \in (A_{\infty})_{\alpha_{\infty}}$, $(c_m^{(1)})_m\geq 0$, $(c_m^{(i)})_m(c_m^{(j)})^*=\delta_{i,j} (c_m^{(1)})^2$, $\|(c_m^{(j)})_m\|$ $=1$, $\| 1- $ $\sum_{j=1}^k (c_m^{(j)})^*$ $(c_m^{(j)}) \|$ $ =1$, $\lim_{m \rightarrow \infty}$ $\tau ({c_m^{(1)}}^m)$ $ = 1/k $, and $\lim_{m\rightarrow \infty}$ $\tau (1 - \sum {c_m^{(j)}}^* c_m^{(j)}) $ $=0$, where $\tau$ is a unique tracial state of $A$. Let $\varepsilon_m > 0$, $m\in \N$ be such that $\varepsilon_m \searrow 0$ and $\tau({c_m^{(1)}}^m) \geq$ $1/k -\varepsilon_m$.
 
Because of the weak Rohlin property of $\alpha \in \Aut (A)$ we obtain $f_m^{(l)}\in A_+^1$, $l, m\in \N$, such that $(f_m^{(l)})_m \in A_{\infty}$ and 
\begin{eqnarray}
& &\alpha_{\infty}^p ((f_m^{(l)})_m) (f_m^{(l)})=0,\ p=1,2,...,l-1, \nonumber \\
& &  \| [f_r^{(l)}, c_m^{(1)} ] \| < \varepsilon_m,\ r\geq m, \nonumber \\
& &\tau ({f_r^{(l)}}^m) > 1/l- \varepsilon_m,\ r\geq m. \nonumber
\end{eqnarray}
Note that any subsequence of $(f_m^{(l)})_m$ satisfies the above conditions.
Since
 $\tau$ is the unique tracial state, taking a subsequence of $(f_m^{(l)})_m $ we may suppose that $\tau ((c_m^{(1)} f_m^{(l)})^m)$ $\approx_{\varepsilon_m}$ $\tau ( {c_m^{(1)}}^m {f_m^{(l)}}^m)$ $\approx_{\varepsilon_m}$ $ \tau ({c_m^{(1)}}^m) \tau({f_m^{(l)}}^m )$ $\geq 1/(kl) - 2\varepsilon_m$. 
Set 
\[g_m^{(l)} = {c_m^{(1)}}^{1/2} f_m^{(l)} {c_m^{(1)}}^{1/2} \in A_+^1,\quad m\in \N, \] then we have that $(g_m^{(l)})_m \in {A_{\infty}}_+^1$, $l\in \N$ and $\displaystyle \liminf_{m\rightarrow \infty}$ $\tau ({g_m ^{(l)}}^m)$ $\geq 1/(kl)$.

By the property (SI) of $A$, we obtain $\ps_m^{(l)} \in A^1$, $m\in \N$, such that 
$(\ps_m^{(l)})_m \in A_{\infty}$, $({\ps_m^{(l)}}^* \ps_m^{(l)})$ $=$ $(1- \sum_{j=1}^k {c_m^{(j)}}^* c_m^{(j)})$, and $ (g_m^{(l)} \ps_m^{(l)})$ $=(\ps_m^{(l)})$.
Remark that $(g_m^{(l)}) =$ $(f_m^{(l)}) (c_m^{(1)}) \leq$ $(f_m^{(l)})$, $(g_m^{(l)}) \leq (c_m^{(1)})$,  $ (f_m^{(l)}) (\ps_m^{(l)})$ $= (\ps_m^{(l)})$, and $(c_m^{(1)})_m (\ps_m^{(l)})$ $=(\ps_m^{(l)})$.
Let $L_n\in \N$ be such that $2L_n^{-1/2} < \varepsilon_n$ ($L_n\nearrow \infty$) and 
define 
\[ s_m^{(L_n)} = L_n^{-1/2} \sum_{p=0}^{L_n -1} \alpha^p (\ps_m^{(L_n)}). \]
Then we have that 
\[\| \alpha (s_m^{(L_n)}) - s_m^{(L_n)} \| \leq 2 L_n^{-1/2} < \varepsilon_n,\quad m\in\N.\]
Let $m_n \in \N$, $n\in \N$ be an increasing sequence with $m_n \nearrow \infty$ such that $(s_{m_n}^{(L_n)})_n \in A_{\infty}$, 
$\|s_{m_n}^{(L_n)} \|$ $\leq 1 + \varepsilon_n$, $\|f_{m_n}^{(L_n)}\ps_{m_n}^{(L_n)}$ $- \ps_{m_n}^{(L_n)} \| <$ $ \varepsilon_n / L_n$, $\| \alpha^p (f_{m_n}^{(L_n)})$ $f_{m_n}^{(L_n)} \|$ $<  \varepsilon_n/ L_n$, $\| {\ps_{m_n}^{(L_n)}}^* \ps_{m_n}^{(L_n)} -$ $(1 - \sum_{j=1}^k {c_{m_n}^{(j)}}^* c_{m_n}^{(j)}) \|$ $ \leq \varepsilon_n$,  $ \| \alpha^p (c_{m_n}^{(j)}) -$ $c_{m_n}^{(j)} \|$ $<  \varepsilon_n /(2k)$, $j=1,2,...,k$, $p=1,2,...,L_n -1$, and $\|c_{m_n}^{(1)} s_{m_n}^{(L_n)} -$ $s_{m_n}^{(L_n)} \|$ $< \varepsilon_n/ L_n$, and set $s_n = s_{m_n}^{(L_n)}$. 

Then we have that $\alpha_{\infty} ((s_n)_n) = (s_n)\in A_{\infty}$,  
\begin{eqnarray} 
s_n^* s_n&\approx_{2\varepsilon_n } & L_n^{-1} ( \sum_{p=0}^{L_n-1} \alpha^p ({\ps_{m_n}^{(L_n)}}^* f_{m_n}^{(L_n)}))(\sum_{q=0}^{L_n-1} \alpha^q (f_{m_n}^{(L_n)} \ps_{m_n}^{(L_n)})) \nonumber \\
&\approx_{\varepsilon_n} & L_n^{-1}\sum_{p=0}^{L_n-1} \alpha^p ({\ps_{m_n}^{(L_n)}}^* {f_{m_n}^{(L_n)}}^2 \ps_{m_n}^{(L_n)}) \nonumber \\
&\approx_{\varepsilon_n} & L_n^{-1}\sum \alpha^p ({\ps_{m_n}^{(L_n)}}^*  \ps_{m_n}^{(L_n)}) \nonumber \\
&\approx_{\varepsilon_n} & L_n^{-1}\sum \alpha^p (1-\sum_{j=1}^k {c_{m_n}^{(j)}}^* c_{m_n}^{(j)}) \nonumber \\
&\approx_{\varepsilon_n} & 1-\sum_{j=1}^k {c_{m_n}^{(j)}}^* c_{m_n}^{(j)}, \nonumber 
\end{eqnarray}   
$\|(s_n)_n\| =1 $, and $(c_{m_n}^{(1)})(s_n) =(s_n)$. Hence we conclude that $\{ (c_{m_n}^{(j)})_n\}_{j=1}^k \cup \{ (s_n)_n \}$ $\subset (A_{\infty})_{\alpha_{\infty}}^1$ and they satisfy the relations $\cR_k$. \eprf

{\it Proof of Theorem \ref{Thm1}}
 Applying Proposition 2.2 in \cite{TW} to $(A\times_{\alpha} \Z)_{\infty}$ it suffices to show the following Lemma.
\begin{lemma}\label{Lem4.7}
Let $A$ be a unital separable projectionless $C^*$-algebra which has a unique tracial state and absorbs the Jiang-Su algebra tensorially. Suppose that $A$ has property (SI) and  $\alpha \in \Aut (A) $ has the weak Rohlin property. Then for any $k\in \N$ there exists a unital embedding $\Phi_k$ of $I(k,k+1)$ into $(A_{\infty})_{\alpha_{\infty}}$.
\end{lemma}
\bprf
By Proposition \ref{Prop4.2} we obtain a set of norm one generators $\{ c_j\}_{j=1}^k \cup \{ s\} $ in $(A_{\infty})_{\alpha_{\infty}}$ satisfying $\cR_k$. 
Since $|s|$ and $c_1$ are norm one elements we have that $\spec (|s|) \supset \{0,1\}$, and since $A_{\infty}$ is a projectionless $C^*$-algebra we have that $\spec (|s|) =[0,1]$. Then, by Corollary \ref{Cor2.1}, we conclude the above statement.
\eprf

\section{Stability for automorphisms }

In this section, using the weak Rohlin property, we show the stability for automorphisms of the Jiang-Su algebra, Theorem \ref{Thms}. 

First, we recall the generalized determinant introduced by P. de la Harpe and G. Skandalis (see \cite{HS}, \cite{KM}, \cite{Pim}). Let $A$ be a unital $C^*$-algebra with a unique tracial state $\tau$. For any piecewise differentiable path $\xi$ $:[0,1]\rightarrow U(A)$, we set 
\[\wtDelta_{\tau}(\xi)= \frac{1}{2\pi\im} \int_0^1 \tau(\dot{\xi}(t)\xi^* (t)) dt \in \R.\]
When $\xi(0)$ $=$ $\xi(1)$ $=1$ we have that $\wtDelta_{\tau} (\xi) \in \tau (K_0(A))$.
For any $u\in U_0(A)$, there exists a piecewise differentiable path $\xi_u:$ $[0,1]\rightarrow U(A)$ such that $\xi_u (0)=1$, $\xi_u(1)=u$. The generalized determinant $\Delta_{\tau}$ associated with the tracial state $\tau$ is the map from $U_0(A)$ to $\R/\tau(K_0(A))$ defined by $\Delta_{\tau}(u)$ $=\wtDelta(\xi_u) + \tau(K_0(A))$. Remark that $\Delta_{\tau}$ is a group homomorphism.

Mimicking the proof of Lemma 6.2 in \cite{KM} we prove the following proposition. Hereinafter, we let $\log$ be the standard branch defined on the complement of the negative real axis.

\begin{proposition}\label{Prop5.1}
Let $B$ be the UHF algebra of rank $k^{\infty}$, where $k$ $\in\N\setminus \{ 1\}$, $\tau$ the unique tracial state of $B$, $\beta\in\Aut (B)$, and $u_n\in U(B)$, $n\in \N$. Suppose that $\beta \in \Aut (B)$ has the Rohlin property and 
\[ \Delta_{\tau} (u_n) =0, \quad {\it for\ any\ }n\in \N.\]
Then there exist $v_n\in U(B)$, $n\in \N$ such that $(v_n)_n\in B_{\infty}$,
\[(v_n\beta(v_n)^*)_n = (u_n)_n, \quad \tau\circ \log (v_n\beta (v_n)^* u_n^*)=0,\ {\it for\ any\ }n\in \N. \]
\end{proposition}

The following lemma was essentially proved in \cite{HS}.
\begin{lemma}\label{Lem5.1}
Let $A$ be a unital $C^*$-algebra with a unique tracial state $\tau$.

\item[(1)]\label{(5.1)}
For $u_1$, $u_2$ $\in U(A)$ with, $\|u_i-1\|<1/2$, $i=1,2$ it follows that 
\[\tau\circ \log (u_1u_2)= \tau \circ \log (u_1) + \tau \circ \log (u_2).\]
\item[(2)]\label{(5.2)}
For $u_1$, $u_2$, and $v$ $\in U(A)$ with $ \| u_1 - u_2 \| < 1/2 $ and $\|v-1\| < 1/4$, it follows that 
\[ \tau \circ \log (u_1 v u_2^* v^*) = \tau \circ \log (u_1 u_2^*).\]
\end{lemma}
\bprf 
{(1)} Let $h_i\in A_{\rm sa}$ be such that $\exp(2\pi \im h_i) =u_i$, $i=1,2$,  and $h_3 \in A_{\rm sa}$ be such that $\exp(2\pi \im h_3) =u_1 u_2$. 
Set $u(t) =$ $\exp (2\pi$ $\im t h_1 )$ $\cdot\exp (2\pi\im  t h_2 )$, $w (t) = \exp (2\pi \im t h_3 )$, $t\in [0,1]$.
Since $\|1-u(t)\|<1$, $\|1-w(t)\| < 1 $, and $\|1-w^* u (t) \| < 2$, $t \in [0,1]$, we can define $h$ $\in C([0,1]) \otimes A_{\rm sa}$ by $h(t)= \log (w^* u(t))$, $t\in [0,1]$, then $u$ and $w$ are homotopic, by $H(s,t) = w (t) \exp((1-s)h(t))$ with fixed endpoints $H(s,0)=1$ and $H(s,1)=w(1)$. Hence, we have that  
\begin{eqnarray}
 \tau \circ \log (u_1u_2) &=& {2\pi \im} \tau (h_3)= \int_{0}^{1} \tau(\dot{w}  w^* (t)) dt = \int_{0}^{1} \tau(\dot{u} u^* (t)) dt \nonumber \\
&=& {2\pi \im} \tau (h_1 + h_2) = \tau \circ \log (u_1) + \tau \circ \log (u_2). \nonumber 
\end{eqnarray}
{(2)} Set $U_1= v^* u_1 v u_1^*$, $U_2 = u_1 u_2^*$, then it follows that $\|U_i -1 \| < 1/2$, $i=$ $1,2$. Applying (1), since $\tau\circ\log(U_1)$ $= \tau\circ\log(v^*)$ $+ \tau\circ\log(u_1vu_1^*)$ $=0$ we have that $\tau\circ\log(U_1U_2)$ $=\tau\circ\log(U_1)$ $+\tau\circ\log(U_2)$ $= \tau \circ\log(U_2)$. 
\eprf

\begin{proof}[proof of Proposition \ref{Prop5.1}] 
Because $\beta \in \Aut (B)$ has the Rohlin property in \cite{Kis1}, there exist $\pv_n \in U(B)$, $n\in\N$ such that $(\pv_n)_n$ $\in B_{\infty}$, and 
\[ (\pv_n\beta(\pv_n)^*)_n = (u_n) .\]
By the assumption and   
\begin{eqnarray}
& &\frac{1}{2\pi\im} \tau \circ \log (\pv_n \beta(\pv_n)^*u_n^*) + \tau (K_0 (B)) \nonumber \\
& &=\Delta_{\tau} (v_n' \beta(v_n')^* u_n^*) = \Delta_{\tau} (v_n'\beta(v_n')^* ) - \Delta_{\tau}(u_n) = - \Delta_{\tau}(u_n), \nonumber
\end{eqnarray}
we have that 
\[\frac{1}{2\pi\im} \tau \circ \log (\pv_n \beta(\pv_n)^*u_n^*)\in \tau(K_0 (B)),\ n\in \N.\]

Since $B$ is the UHF algebra of rank $k^{\infty}$, we obtain $l_n\in \N$ and $m_n\in \Z$ such that $(m_n,k)=1$ and 
\[k^{-l_n}m_n = -\frac{1}{2\pi\im} \tau \circ \log (\pv_n\beta(\pv_n)^*u_n^*) \in \tau (K_0 (B)).\]
Set $\lambda_n$ $=$ $\exp (2\pi \im k^{-l_n} m_n)$, then we have that $\lambda_n \rightarrow 1$, by $(\pv_n$ $\beta(\pv_n)^*$ $u_n^*)_n=1$. By the Rohlin property of $\beta \in \Aut (B)$, there exist $p_n \in P(B)$ and $z_n \in U(B)$, $n\in \N$ such that $(p_n)_n \in B_{\infty}$, $(z_n)_n=1_{B_{\infty}}$, and 
\[ \sum_{j=0}^{k^{l_n}-1}
(\Ad z_n\circ \beta)^j (p_n) = 1_{B}. \]
Define
\begin{eqnarray}
\ov_n &=& \sum_{j=0}^{k^{l_n}-1} \exp (2\pi\im jk^{-l_n}m_n) \cdot (\Ad z_n\circ \beta)^j (p_n), \nonumber \\
v_n &=& \pv_n \ov_n \in U(B),\ n\in \N.\nonumber
\end{eqnarray}
Taking a subsequence of $(p_n)_n$, we may suppose that $(\ov_n)_n$ $\in B_{\infty}$. Then it follows that $(v_n)_n$ $\in B_{\infty}$. 
By the definition of $\ov_n$ we have that $\ov_n\Ad z_n \circ \beta(\ov_n)^*$ $=\lambda_n$ and 

\[(v_n\beta(v_n)^* u_n^*)_n=(v_n\Ad z_n \circ \beta(v_n^*)u_n^*)_n=  (\lambda_n \pv_n\beta(\pv_n)^* u_n^*)_n=1. \]
And, by Lemma \ref{Lem5.1}, we have that $\tau\circ\log(v_n\beta(v_n)^*u_n^*)$
\begin{eqnarray}
&=& \tau\circ\log(v_n\Ad z_n\circ \beta(v_n)^* u_n^*) \nonumber \\ 
&=& \tau \circ \log(\ov_n \Ad z_n \circ \beta (\ov_n)^* \Ad z_n \circ \beta(\pv_n)^* u_n^* \pv_n) \nonumber \\
&=& 2\pi\im k^{-l_n}m_n + \tau \circ \log(\pv_n \beta (\pv_n)^*u_n^*)=0,\ n \in \N.  \nonumber
\end{eqnarray}
\end{proof}

\begin{theorem}\label{Thms}
Suppose that $\sigma \in \Aut(\Zj)$ has the weak Rohlin property and $u_n\in U(\Zj)$, $n\in \N$ satisfy $(u_n)_n \in \Zj_{\infty}$ and that
\[\Delta_{\tau_{\Zj}}(u_n) =0,\ {\it for\ any\ } n\in\N. \]
Then
 there exist $v_n \in U(\Zj)$, $n\in \N$ such that $(v_n)_n \in \Zj_{\infty}$ and 
\[(v_n \sigma(v_n)^*)_n = (u_n)_n.\]
\end{theorem}

The following lemma is a direct adaptation of Proposition 4.6 in \cite{KM}.
\begin{lemma}\label{Lem5.4}
For any $c>0$ there exists $c'>0$ such that the following holds. 
Let $B$ be a UHF algebra, $\tau$ the unique tracial state of $B$. Suppose that $\wtu_n$ $\in U(C([0,1])\otimes B)$, $n\in \N$ satisfies that $(\wtu_n)_n \in (C([0,1])\otimes B)_{\infty} $, $(\wtu_n (i))_n=1$, $i=0,1$, 
\[\wtDelta_{\tau} (\wtu_n)=0,\quad \tau\circ\log (\wtu_n(i))=0,\ i=0,1,\ n\in \N,\]
 and $\Lip (\wtu_n) <c$, $n\in \N$. Then there exist $y_n \in U(C([0,1]^2)\otimes B )$, $n\in \N$ such that $(y_n)_n \in (C([0,1]^2)\otimes B)_{\infty}$, 
\[y_n(0,t) =1_B,\quad y_n(1,t)=\wtu_n(t),\ t\in [0,1],\]
\[ y_n(s,i)=\exp(\log(\wtu_n(i))s),\ i=0,1,\ s\in[0,1],\]
 and $\Lip(y_n)<c'$, $n\in\N$. 
\end{lemma}
\bprf
Set $\partial E =\{ (s,t)\in[0,1]^2 ;\ \{s,t\}\cap \{0,1\} \neq \phi\}$. By Proposition 4.6 in \cite{KM}, for $c>0$, we obtain $c'>0$  satisfying that: 
for any AF-algebra $A$ and for any $z$ $\in U(C(\partial E)\otimes A)$ with $z(0,0)=1$, $\Lip (z) <c$, and $[z]_1=0 \in K_1(C(\partial (E))\otimes A)$, there exist $\wtz$ $\in U(C([0,1]^2)\otimes A)$ such that $\wtz|_{\partial E}=z$ and $\Lip(\wtz)< c'$. Suppose that $\wtu_n$ $\in U(C([0,1])\otimes B)$ satisfies the conditions in the lemma. Define $U_n\in$ $U(C(\partial E)\otimes B)$ by 
\begin{eqnarray}
U_n(s,t)=\left\{ \begin{array}{lll}
1,\quad & s=0, \\
\wtu_n(t), \quad  &s=1, \\
\exp(\log(\wtu_n(i))s), \quad &t=i,\ i=0,1. \\
\end{array} \right.\nonumber
\end{eqnarray} 
Then we have that $\Lip(U_n) < c$ for any $n\in \N$.
By the assumption, regarding $U_n$ $\in U(C(\T)\otimes B)$, we have that $[U_n]_1=$ $\wtDelta_{\tau} (U_n)=0 $ in $\tau(K_0(B))$.

Let $B_n$, $n\in \N$ be an increasing sequence of  matrix subalgebras of $B$ with $1_{B_n}=1_B$ and $\overline{\bigcup B_n}=B$.
Since $(U_n)_n$ $\in U((C(\partial E)\otimes B)_{\infty})$, slightly modyfying $U_n$, we obtain an increasing sequence $m_n\in \N$, $n\in \N$ and $U_n'\in$ $U(C(\partial E)\otimes (B_{m_n}'\cap B))$ such that $m_n\nearrow \infty$, $(U_n')_n=(U_n)_n$, $U_n'(0,0)=1$, and $\Lip(U_n') < c$.
Since $B_{m_n}'\cap B$ has the unique tracial state $\tau|_{B_{m_n}'\cap B}$, it follows that $[U_n']$ $_{K_1(B_{m_n}'\cap B)}$ $=$ $\wtDelta_{\tau_{B_{m_n}'\cap B}}(U_n')$ $=$ $\wtDelta_{\tau_B}(U_n')=$ $\wtDelta_{\tau_B}(U_n)=0$, then we obtain $\wtU_n$ $\in U(C([0,1]^2)\otimes(B_{m_n}'\cap B))$, $n\in \N$ such that $\wtU_n|_{\partial E} = U_n'$ and $\Lip (\wtU_n) < c'$.
Then we have that $(\wtU_n)_n$ $\in(C([0,1]^2)\otimes B)_{\infty}$.  Since $\wtU_n|_{\partial E} =U_n'$, $n\in\N$, slightly modyfying $\wtU_n$ on $\partial E$, we obtain $y_n$ $\in U(C([0,1]^2)\otimes B)$, $n\in \N$ and $\varepsilon >0$ such that $(y_n)_n= (\wtU_n)_n$, $y_n|_{\partial E} =U_n$, and $\Lip (y_n) < c'+\varepsilon$ for any $n\in \N$.
\eprf

 As in the proof of Proposition 2.2 in \cite{TW}, the embedding $I(k,k+1) \subset_{\rm unital} (A_{\infty})_{\alpha_{\infty}}$ obtained in Lemma \ref{Lem4.7} implies  the following Lemma.
\begin{lemma}\label{Lem5.2}
Let $A$ be a a unital separable projectionless $C^*$-algebra which has a unique tracial state and absorbs the Jiang-Su algebra tensorially. Suppose that $A$ has property (SI) and  $\alpha $ is an automorphism of $A$ with the weak Rohlin property. Then there exists a unital embedding of $\Zj$ into $(A_{\infty})_{\alpha_{\infty}}$.
\end{lemma}

\begin{proof}[proof of Theorem \ref{Thms}] Let $\Zjk$, $B^{(i)}$, and $\Phi^{(i)}$, $i=0,1$, be the projectionless $C^*$-algebra, the UHF algebras, and the unital embeddings in the proof of Proposition \ref{Prop3.4}.

Because of $\Zj\otimes B^{(i)} \cong B^{(i)}$, the Rohlin property of $\sigma\otimes \id_{B^{(i)}} \in \Aut(\Zj\otimes B^{(i)})$, and 
\[ \Delta_{\tau_{\Zj\otimes B^{(i)}}}(u_n\otimes 1_{B^{(i)}}) =0 ,\ {\rm in }\ \R/\tau(K_0(B^{(i)})), \ n\in \N, \]
 applying Proposition \ref{Prop5.1}, we obtain ${V_n}^{(i)}\in U(\Zj \otimes B^{(i)})$, $i=0,1$, $n\in \N$ such that $({V_n}^{(i)})_n \in (\Zj \otimes B^{(i)})_{\infty}$ and 
\begin{eqnarray}
 & &({V_n}^{(i)} \sigma\otimes \id_{B^{(i)}} ({V_n}^{(i)})^*)_n = (u_n \otimes 1_{B^{(i)}})_n,\ i=0,1, \nonumber \\
 & & \tau_{\Zj\otimes B^{(i)}} \circ \log (V_n^{(i)} \sigma\otimes \id_{B^{(i)}}(V_n^{(i)})^* u_n^*\otimes 1_{B^{(i)}}) = 0, \ n\in \N. \nonumber 
\end{eqnarray}
By the following argument, we would like to obtain a path of unitaries $\wtv_n$ in $\Zj \otimes \Zjk$ with endpoints $\Phi^{(i)}(V_n^{(i)})$ $\in U( \Zj\otimes B^{(0)}\otimes B^{(1)})$, $i=0,1$ which satisfies $\wtv_n (t) \sigma\otimes\id_{B^{(0)}\otimes B^{(1)}}$ $(\wtv_n(t)^*)\approx $ $u_n\otimes 1_{B^{(0)}\otimes B^{(1)}}$ for any $t\in $ $[0,1]$.

Set 
\begin{eqnarray}
U_{n,1}&=& \Phi^{(0)} (V_n^{(0)})^* \Phi^{(1)}(V_n^{(1)}), \nonumber \\
W_n &=& U_{n,1}\sigma\otimes \id_{ B^{(0)}\otimes B^{(1)}} (U_{n,1})^*,\ n\in \N. \nonumber 
\end{eqnarray}
Then it follows that $(U_{n,1})_n$ $\in (\Zj \otimes B^{(0)}\otimes B^{(1)})_{\infty}$, $(W_n)_n=1_{(\Zj\otimes B^{(0)}\otimes B^{(1)})_{\infty}}$, and, by (1) in Proposition \ref{Lem5.1}, $\tau_{\Zj\otimes B^{(0)}\otimes B^{(1)}} \circ\log (W_n)$
\begin{eqnarray}
&=&\tau\circ\log (\Phi^{(1)}(V_n^{(1)}\sigma\otimes \id_{B^{(1)}}(V_n^{(1)})^* u_n^*\otimes 1_{ B^{(1)}}) \nonumber \\
& & \cdot\Phi^{(0)}(V_n^{(0)}\sigma\otimes \id_{B^{(0)}} (V_n^{(0)})^*u_n^*\otimes 1_{B^{(0)}})^*) \nonumber \\
&=&\sum_{i=0,1} (-1)^{1-i}\tau\circ\log (\Phi^{(i)} (V_n^{(i)} \sigma\otimes\id_{B^{(i)}} (V_n^{(i)})^*)u_n^*\otimes 1_{B^{(i)}})=0,  \nonumber 
\end{eqnarray}
for any $n\in \N$.

Since $(U_{n,1})_n \in (\Zj\otimes B^{(0)}\otimes B^{(1)})_{\infty}$, there exist $\tU_n \in$ $U(C([0,1])\otimes \Zj \otimes B^{(0)} \otimes B^{(1)})$, $n\in\N$ be such that $\tU_n(0)=1$, $\tU_n(1) =U_{n,1}$, $(\tU_n)_n \in$ $(C([0,1])\otimes \Zj \otimes B^{(0)}\otimes B^{(1)})_{\infty}$, and $\Lip (\tU_n) < \pi + \varepsilon$ for some $\varepsilon >0$. Define $\tT_n^{(j)} \in $ $U(C([0,1])\otimes \Zj \otimes B^{(0)} \otimes B^{(1)})$, $j$, $n\in \N$ by 
\[\tT_n^{(j)} = \tU_n \id_{C([0,1])}\otimes\sigma^j\otimes \id_{B^{(0)}\otimes B^{(1)}} (\tU_n )^*,  \] 
$\tT_n^{(0)} =1$. Note that 
\[ \tT_n^{(j)} \id \otimes \sigma\otimes\id (\tT_n^{(j-1)})^* = \tT_n^{(1)},\quad j, n\in\N.\]
By $(\tU_n)_n \in (C([0,1])\otimes \Zj \otimes B^{(0)}\otimes B^{(1)})_{\infty}$, $(W_n)_n=1$, $\tau \circ \log (W_n) =0$, and $\Lip(\wtU_n)< \pi +\varepsilon$, we have that $(\tT_n^{(j)})_n \in (C([0,1])\otimes\Zj\otimes B^{(0)}\otimes B^{(1)})_{\infty}$, $(\tT_n^{(j)}(1))_n=1$, $\tau\circ\log (\tT_n^{(j)}(1))=$ $j\tau\circ\log(W_n)$ $=0$, and $\Lip (\tT_n^{(j)})< 2(\pi +\varepsilon)$, $j\in \N$. Then, by Lemma \ref{Lem5.4}, we obtain a constant $c>0$ and $y_n^{(j)}\in U(C([0,1]^2)\otimes \Zj\otimes B^{(0)}\otimes B^{(1)})$, $j\in\N$ such that $(y_n^{(j)})_n\in(C([0,1]^2)\otimes \Zj\otimes B^{(0)}\otimes B^{(1)})_{\infty}$,
\[ y_n^{(j)}(0,t)=1, \quad y_n^{(j)}(1,t)=\tT_n^{(j)} (t),\ t\in [0,1]\]
\[ y_n^{(j)} (s,1)=\exp(\log(\tT_n^{(j)}(1)s),\quad y_n^{(j)}(s,0)=1,\ s\in [0,1], \]
and $\Lip (y_n^{(j)})< c$, $n\in\N$.

By the Rohlin property of $\sigma\otimes\id_{B^{(0)}\otimes B^{(1)}}$ we obtain $p_m^{(l)} \in$ $P(\Zj\otimes B^{(0)}\otimes B^{(1)})$ and $z_m^{(l)}\in$ $U(\Zj \otimes B^{(0)}\otimes B^{(1)})$, $l$, $m\in \N$ such that $(p_m^{(l)})_m \in$   $(\Zj\otimes B^{(0)}$ $\otimes B^{(1)})_{\infty}$, $(z_m^{(l)})_m=1$, and 
\[  \sum_{j=0}^{k^l-1} (\Ad z_m^{(l)}\circ \sigma \otimes \id)^j (p_m^{(l)}) =1.\]
Set $y_n^{(j)}(s)(t)$ $= y_n^{(j)}(s,t)$, $s, t\in[0,1]$, $j$, $n \in \N$,
\begin{eqnarray}
\tsigma_m^{(l)} &=& \id_{C([0,1])}\otimes (\Ad z_m^{(l)} \circ \sigma \otimes \id_{B^{(0)}\otimes B^{(1)}}),\quad   l,\ m\in \N, \nonumber \\
{\tW }_{l,m,n}' &=& \sum_{j=0}^{k^l -1} \tT_n^{(j)} \cdot (\tsigma_m^{(l)} )^{j-k^l}(y_n^{(k^l)} (j/k^l))^* \cdot (\tsigma_m^{(l)})^j (1_{C([0,1])}\otimes p_m^{(l)} ) .\nonumber 
\end{eqnarray}
Since  $((\tsigma_m^{(l)})^j$ $(1_{C([0,1])}\otimes p_m^{(l)}))_m$ $\in$ $(C([0,1])\otimes\Zj$ $\otimes B^{(0)}$ $\otimes B^{(1)})_{\infty}$, $j$ $=$ $0$,$1$,$...$,$k^l-1$ are mutually orthogonal projections, we have that $({\tW}_{l,m,n}')_m$ is a unitary and obtain $\tW_{l,m,n} \in U(C([0,1])\otimes \Zj \otimes B^{(0)}\otimes B^{(1)})$, $l$, $m$, $n \in \N$ such that $(\tW_{l,m,n})_m = ({\tW}_{l,m,n}')_m$.
By the definition of $\tW'_{l,m,n}$,  we have that 
\[ \|(\tW_{l,m,n})_m\cdot ( \id_{C([0,1])}\otimes \sigma\otimes\id (\tW_{l,m,n})^*)_m-({\tT}_n^{(1)})_m \| < c /k^l, \quad l,\ n\in \N. \]
Since $(\tT_n^{(j)})_n$, $((\tsigma_m^{(l)})^{j-k^l}(y_n^{(k^l)}(j/k^l)))_n$,
 and $(\tsigma_m^{(l)})^j(1_{C([0,1])}\otimes p_m^{(l)})_m$ $\in (C([0,1])$ $\otimes\Zj\otimes B^{(0)}\otimes B^{(1)})_{\infty}$, and $\|1-W_n\|\rightarrow 0$, we obtain a slow increasing sequence $l_n$, $n\in\N$ and a fast increasing sequence $m_n\in\N$, $n\in\N$ such that $l_n\nearrow \infty$, $m_n\nearrow \infty$, 
\[(\tW_{l_n,m_n,n})_n \in (C([0,1])\otimes\Zj\otimes B^{(0)}\otimes B^{(1)})_{\infty},\] 
\[ k^{2l_n}\|1-W_n\| \rightarrow 0,\quad \|\tW_{l_n,m_n,n} \id\otimes \sigma \otimes\id(\tW_{l_n,m_n,n})^* - \tT_n^{(1)} \| < c/k^{l_n}. \]

Set \[ \tV_n' = {\tW}_{l_n,m_n,n}^* \tU_n \in U(C([0,1])\otimes \Zj\otimes B^{(0)}\otimes B^{(1)}), \quad n\in\N. \]
Then it follows that $(\tV_n')_n \in C([0,1])\otimes \Zj\otimes B^{(0)}\otimes B^{(1)})_{\infty}$ and
\begin{eqnarray}
& & (\tV_n'  \id_{C([0,1])}\otimes \sigma \otimes \id_{B^{(0)}\otimes B^{(1)}} (\tV_n' )^*)_n \nonumber \\
& & =(\Ad {\tW}_{l_n,m_n,n}^* ({\tT}_n^{(1)}\cdot \id \otimes \sigma \otimes \id (\tW_{l_n,m_n,n})\cdot {\tW}_{l_n,m_n,n}^*))_n=1,\nonumber 
\end{eqnarray}
Since   $\|1- \tT_n^{(j)}(1)\| \leq j \| 1- W_n\|$ $\rightarrow 0$, and $\|1-y_n^{(k^{l_n})}(i/k^{l_n},1)\|$ $\leq \|1- \tT_n^{(k^{l_n})} (1)\|$, it follows that
\begin{eqnarray}
\tW_{l_n,m_n,n}'(1)  &=&  \sum_{j=0}^{k^{l_n} -1} \tT_n^{(j)} (1) (\tsigma_{m_n}^{(l_n)})^{j-k^{l_n}} (y_n^{(k^{l_n})}(j/k^{l_n},1))^* (\tsigma_{m_n}^{(l_n)})^j (p_{m_n}^{(l_n)}) \nonumber \\
&\approx_{\delta_n} & 1. \nonumber 
\end{eqnarray}
where $\delta_n =2k^{2 l_n} \| 1- W_n\|$, $n\in \N$, and then we have that 
\[(\tV_n' (1))_n = (U_{n,1})_n=(\Phi^{(0)}(V_n^{(0)})^* \Phi^{(1)}(V_n^{(1)}))_n. \]

Define  
\[ \tV_n (t) = \Phi^{(0)} (V_n^{(0)}) \tV_n' (t),\quad  t\in [0,1]. \]
Then we have that $(\tV_n)_n \in (C([0,1])\otimes \Zj\otimes B^{(0)}\otimes B^{(1)})_{\infty}$
\[(\tV_n(i))_n  = (\Phi^{(i)} (V_n^{(i)}))_n,\ i=0,1, \]
and that $(\tV_n \id_{C([0,1])}\otimes \sigma \otimes \id_{B^{(0)}\otimes B^{(1)}} (\tV_n)^*)_n$
\begin{eqnarray} 
 &=& (1_{C([0,1])}\otimes \Phi^{(0)}(V_n^{(0)}\sigma\otimes \id_{B^{(0)}} (V_n^{(0)})^*))_n \nonumber \\
 &=& (1_{C([0,1])}\otimes u_n \otimes 1_{B^{(0)}\otimes B^{(1)}} )_n. \nonumber
\end{eqnarray}
 Slightly modyfying $\tV_n$ at the end points, we obtain $\wtv_n\in $ $U(\Zj\otimes \Zjk)$ such that $(\wtv_n)_n= $ $(\tV_n)_n$,
\[ \wtv_n(i) = \Phi^{(i)}(V_n^{(i)}),\ i=0,1,\quad (\wtv_n \sigma \otimes \id_{\Zjk}(\wtv_n)^*)_n = ( u_n \otimes 1_{\Zjk})_n.\]  

Finally, we obtain $(v_n)_n$ $\in \Zj_{\infty}$ which corresponds to $\wtv_n$ $\in \Zj\otimes\Zj_k$ and satisfies $(v_n\sigma (v_n^*))_n =(u_n)_n$, by the following.
By Lemma \ref{Lem5.2} and $\Zjk \subset_{\rm unital} \Zj$, we obtain a unital embedding $\Psi: \Zj \otimes\Zjk$ $\hookrightarrow \Zj^{\infty}$ such that $\sigma_{\infty} \circ \Psi = \Psi \circ \sigma \otimes \id_{\Zjk}$ and $\Psi(a\otimes 1_{\Zjk})=$ $ a \in$ $ \Zj\subset$ $ \Zj^{\infty}$, $a\in \Zj$. 
Let $F_n\subset \Zj^1$, $n\in \N$ be an increasing sequence of finite subsets of $\Zj^1$ and $\varepsilon_n>0$, $n\in\N$ a decreasing sequence such that $\overline{\bigcup F_n} =\Zj^1$, $\varepsilon_n \searrow 0$, 
\[\| [\wtv_n, x\otimes 1_{\Zjk} ] \| < \varepsilon_n , \quad x\in F_n .\] 
\[\|\wtv_n \sigma\otimes \id_{\Zjk} (\wtv_n)^* -  u_n\otimes 1_{\Zjk}\| < \varepsilon_n, \]
It follows that $\| [ \Psi(\wtv_n), x] \|$ $=\|\Psi ([\wtv_n, x\otimes 1_{\Zjk}]) \|$ $<\varepsilon_n$, $x\in F_n$, and 
\begin{eqnarray} 
\Psi(\wtv_n) \sigma_{\infty}(\Psi(\wtv_n)^*)&=& \Psi(\wtv_n \sigma\otimes \id_{\Zjk}(\wtv_n)^*) \nonumber \\
&\approx_{\varepsilon_n} & \Psi ( u_n \otimes 1_{\Zjk}) =u_n. \nonumber 
\end{eqnarray}
Denote by $v_{n,p} \in U(\Zj)$, $p \in \N$ components of $\Psi(\wtv_n) \in$ $U(\Zj^{\infty})$, then we obtain an increasing sequence $p_n \in \N$, $n\in\N$ such that 
\[ (v_{n,p_n} \sigma(v_{n,p_n})^*)_n = (u_n)_n,\quad (v_{n,p_n})_n\in \Zj_{\infty}.\]
Define $v_n=v_{n,p_n}$. This completes the proof. 
\end{proof}
\begin{proof}[proof of Theorem \ref{Thm2}]
By using the stability of the following form instead of Proposition 4.3 \cite{Kis2} and by the Evans-Kishimoto intertwining argument in the proof of Theorem 5.1 \cite{Kis2} we can give the proof, where we omit the detail.
\begin{corollary}
Suppose that $\sigma \in \Aut (\Zj)$ has the weak Rohlin property. For any finite subset $F$ of $\Zj^1$ and $\varepsilon >0$, there exists a finite subset $G$ of $\Zj^1$ and $\delta >0$ satisfying that: for any $u\in U(\Zj)$ with $\|[u,y]\| <\delta$, $y\in G$, there exist $v \in U(\Zj)$ and $\lambda \in \T$ such that 
\[ \|v \sigma(v)^* -\lambda u \|<\varepsilon,\quad \|[v,x]\| < \varepsilon,\ x\in F.\]
\end{corollary}
\bprf
For $u_n\in U(\Zj)$, $n\in\N$ with $(u_n)_n\in \Zj_{\infty}$, set $\lambda_n$ $=\exp(-2\pi \im$ $\Delta_{\tau_{\Zj}}(u_n))$ $\in\T$. Since $\Delta_{\tau_{\Zj}}(\lambda_n u_n)=0$ $\in \R/\tau(K_0(\Zj))$, by the above theorem we obtain $v_n \in U(\Zj)$, $n\in \N$ such that $(v_n)_n \in \Zj_{\infty}$ and 
\[(v_n \sigma(v_n)^*)_n = (\lambda_nu_n)_n.\]

Assume that there exist a finite subset $F$ of $\Zj^1$ and $\varepsilon>0$ satisfying that: For any finite subset $G$ of $\Zj^1$ and $\delta>0$ there exists $u\in U(\Zj)$ with $\|[u, y]\|< \delta$, $y\in G$   
 such that if $v\in U(\Zj)$ and $\lambda \in \T$ satisfy $\|v \sigma(v)^*-\lambda u\| < \varepsilon $ then $\|[v, x] \| \geq \varepsilon$ for some $x\in F$.
This contradicts  the above statement.
\end{proof}

{\bf Acknowledgments.}
The author would like to thank his advisor Akitaka Kishimoto
 for many valuable discussions, and he also would like to thank Hiroki Matui for the preprint \cite{Mat2} and useful advices.

\end{document}